\documentclass[12pt]{amsart}
\usepackage{graphicx,amssymb,stmaryrd,amsfonts,epsfig,amsthm,a4,amsmath,mathrsfs,url}
\usepackage{vmargin,amscd}
\usepackage{eufrak}

\newtheorem{thm}{Theorem}[section]
\newtheorem{cor}[thm]{Corollary}
\newtheorem{lem}[thm]{Lemma}
\newtheorem{prop}[thm]{Proposition}
\theoremstyle{definition}
\newtheorem{defn}[thm]{Definition}

\newtheorem{exe}[thm]{Example}
\newtheorem{rem}[thm]{Remark}
\theoremstyle{remark}

\numberwithin{equation}{section}

\newcommand{\SL}{\textnormal{SL}}

\newcommand{\Z}{\mathbf{Z}}

\newcommand{\R}{\mathbf{R}}

\newcommand{\Q}{\mathbf{Q}}
\newcommand{\K}{\mathbf{K}}

\newcommand{\tpr}{\begin{tiny}\noindent Proof:}

\title{Compactly presented groups}
\author{Yves de Cornulier}%
\address{IRMAR \\ Campus de Beaulieu \\
35042 Rennes Cedex, France}
\email{yves.decornulier@univ-rennes1.fr}
\date{\today}

\begin{document}

\maketitle



%



\begin{abstract}
This survey purports to be an elementary introduction to compactly presented groups, which are the analogue of finitely presented groups in the broader realm of locally compact groups. In particular, compact presentation is interpreted as a coarse simple connectedness condition on the Cayley graph, and in particular is a quasi-isometry invariant.

In the appendix, an example of a Lie group, not quasi-isometric to any homogeneous graph, is given; the short argument relies on results of Trofimov and Pansu, anterior to~1990. 
\end{abstract}

\section{Introduction}

In geometric group theory, the principal object of study is often a discrete finitely generated group $G$. Its geometry is the geometry of its Cayley graph $\Gamma(G,S)$, where $S$ is a finite generating set, whose vertices are elements of $G$ and edges are of the form $\{g,gs\}$ for $g\in G,s\in S\cup S^{-1}-\{1\}$. It has long been known that the study of the geometry of those groups can be eased when $G$ sits as a cocompact lattice in a non-discrete locally compact group, like a a connected Lie group. This holds, for instance, when $G$ is a finitely generated torsion-free nilpotent group, or more generally a polycyclic group (see \cite[Chapters 2--4]{Rag}). For some classes of solvable groups, for instance a matrix group over $\Z[1/n]$, the target group has to be a linear group over a product of local fields. It thus appears that the natural setting for the geometric study of groups is to consider compactly generated locally compact groups, the compact generation reducing to finite generation in the case of discrete groups. In the locally compact setting, the Cayley graph with respect to a compact generating subset appears at first sight as awful as for instance it has infinite degree; however its large scale geometry has a reasonable behaviour: for instance, it is always quasi-isometric to a graph of bounded degree (obtained by restricting to a ``separated net", see \cite[\S 1.A]{Gromov93}), although this graph cannot always be chosen homogeneous (see Appendix \ref{secho}).

Compactly presented groups generalize to locally compact groups what finitely presented groups are to discrete groups. On the other hand they are far less well-known\footnote{A Google search on February 2, 2009 gives 13500 occurences for ``finitely presented group" and only 5 for ``compactly presented group", including quotation marks in both cases.}, although they were introduced and studied by the German School in the sixties and seventies \cite{Kn,Behr,Ab1}.


A locally compact group $G$ is {\it compactly presented} if it has a compact subset $S$ such that $G$, as an abstract group, has a presentation with $S$ as a set of generators, and relators of bounded length. For instance, a discrete group is compactly presented if and only if it is finitely presented.

Other compactly presented locally compact groups include connected Lie groups, reductive algebraic groups over local fields, but not all algebraic groups over local field: for instance, if $\K$ is any non-Archimedean local field, then the compactly generated group $\SL_2(\K)\ltimes\K^2$ is not compactly presented (see Example \ref{slh}), although compactly generated. An extensive and thorough study led Abels \cite{Ab} to characterize linear algebraic groups over a local field of characteristic zero that are compactly presented. In contrast, the case of finite characteristic is not settled yet.


The following note intends to describe the basics on compactly presented groups, with the hope to make them better known. Properties (\ref{e1}), (\ref{e2}), (\ref{e4}) below were obtained in \cite{Ab1}; we obtain here these results with in mind the geometric point of view initiated by Gromov.

\begin{enumerate}
\item\label{e1} Being compactly presented is preserved by extensions, and by quotients by closed normal subgroups that are compactly generated as normal subgroups (Lemmas \ref {ext} and \ref{lem:bp_to_quot}).
\item\label{e2} Any compactly generated locally compact group has a covering which is a compactly presented locally compact group (a covering of $G$ means a locally compact group $H$ whose quotient by some normal discrete subgroup is isomorphic to $G$) (Proposition \ref{cov}).
\item\label{e3} Among compactly generated locally compact groups, to be compactly presented is a quasi-isometry invariant (Corollary \ref{qinv}). In particular, it is inherited by and from closed cocompact subgroups.
\item\label{e4} If, in an exact sequence of locally compact groups $1\to N\to G\to Q\to 1$, the group $Q$ is compactly presented and $G$ is compactly generated, then $N$ is compactly generated as a normal subgroup of $G$. In particular, if $N$ is central, then it is compactly generated. (Proposition \ref{prop:ker_comp_pres_quot})
\end{enumerate}

Section~\ref{bp} introduces the new and more general notion of groups {\it boundedly generated} by some subset. Compactly presented groups are introduced in Section~\ref{cpg}.
In Section~\ref{rips}, we explain the link between simple connectedness and compact presentedness, which notably allows to obtain the two latter results (\ref{e3}) and (\ref{e4}). Some further examples are developed in Section~\ref{appex}.

\section{Boundedly presenting subsets}\label{bp}

\begin{defn}
Let $G$ be a group and $S\subset G$ a subset. We say that $G$ is
\textit{boundedly presented} by $S$ if $G$ has a presentation with
$S$ as set of generators, and relations of bounded length.
\end{defn}~

\begin{exe}~
\begin{itemize}
\item If $S$ is finite, then $G$ is finitely presented if and only
if $G$ is boundedly presented by $S$.

\item The group $G$ is always boundedly presented by itself.
Indeed, as set of relations one can choose relations of the form
$gh=k$, $g,h,k\in G$, which have length 3. \end{itemize}
\end{exe}

\begin{defn}
Let $G$ be a group and $S$ a generating subset. We say that $S$ is
a \textit{defining} (generating) subset if $G$ is presented with
$S$ as set of generators, subject to the relations $gh=k$ for
$g,h,k\in S$.
\end{defn}

Clearly, if $S$ is a defining generating subset of $G$, then $G$
is boundedly presented by $S$. This has a weak converse
\cite[Hilfssatz 1]{Behr}. Write $S^n=\{s_1\dots
s_n:\;s_1,\dots,s_n\in S\}$.

\begin{lem}
Let $G$ be boundedly presented by a subset $S$. Then, for some
$n$, $S^n$ is a defining generating subset of
$G$.\label{lem:boun_pres_def_gen}
\end{lem}
\begin{proof} Let $G$ be presented by $S$ as set of generators, with
relations of length bounded by $n$. We claim that $S^n$ is a
defining generating subset of $G$.

Now let us work in the group $\tilde{G}$ presented with $S^n$ as
set of generators and subject to the relations of the form $gh=k$
for $g,h,k\in S^n$ whenever this equality holds in~$G$. Let
$u_1,\dots,u_m$ be elements in $S^n$, such that the relation
$u_1\dots u_m=1$ holds in $G$. Write $u_i=v_{i1}\dots v_{in}$ in
$G$, with all~$v_{ij}\in S$. Then it is clear from the defining
relations of $\tilde{G}$ that $u_i=v_{i1}\dots v_{in}$ holds in
$\tilde{G}$, so that $u_1\dots u_m=v_{11}\dots v_{1n}\dots
v_{m1}\dots v_{mn}$ holds in $\tilde{G}$. Since $v_{11}\dots
v_{1n}\dots v_{m1}\dots v_{mn}=1$ holds in $G$ and by the
assumption on $S$, this element can be written (in the free group
generated by $S$, hence in $\tilde{G}$) as a product of conjugates
of elements $w_k$ of length $\le n$ with respect to~$S$, such that
$w_k=1$ in $G$. Again it follows from the defining relations of
$\tilde{G}$ that the equalities $w_k=1$ hold in $\tilde{G}$.
Accordingly $u_1\dots u_m=1$ holds in $\tilde{G}$. This shows that
the natural morphism of $\tilde{G}$ onto $G$ is bijective, so that
the lemma is proved.\end{proof}


\begin{lem}
Let $G$ be a group and $S$ a generating subset. Then $G$ is
boundedly presented by $S$ if and only if $G$ is boundedly
presented by $S\cup S^{-1}$.
\end{lem}
\begin{proof} Immediate from the definition.\end{proof}

\begin{lem}
Let $G$ be a group, and $S_1,S_2$ be generating subsets.

(1) If $G$ is boundedly presented by $S_1^n$ for some $n\ge 1$,
then it is boundedly presented by $S_1$

(2) If $S_1\subset S_2\subset S_1^n$ for some $n\ge 1$, and if $G$
is boundedly presented by $S_1$, then it is boundedly presented by
$S_2$.

(3) If $S_1^m\subset S_2^n\subset S_1^{m'}$ for some $m,m',n\ge
1$, then $G$ is boundedly presented by $S_1$ if and only if it is
boundedly presented by $S_2$.
\end{lem}
\begin{proof} (1) is obvious: take relators with bounded length with
respect to generators in $S_1^n$, and, expressing elements of
$S_1^n$ as products of $n$ elements of $S_1$, we get defining
relators of bounded length with respect to generators in $S_1$.

(2) Consider a family $(r_i)$ of relations with bounded length
with respect to generators in $S_1$, defining a presentation of
$G$. Consider the group $\tilde{G}$ presented with $S_2$ as set of
generators, and subject to the relations on the one hand $(r_i)$
and on the other hand of the form $u_1\dots u_n=u$ for
$x_1,\dots,x_n\in S_1$ and $u\in S_2$. We claim that the natural
morphism of $\tilde{G}$ onto $G$ is injective.

Consider an element $u_1\dots u_k$ of $\tilde{G}$, with $u_i\in
S_2$, and suppose that $u_1\dots u_k=1$ holds in $G$. Write $u_i$
as a word of length $\le n$ in elements of $S_1$. Then the fact
that $u_1\dots u_k=1$ in $\tilde{G}$ follows from the relations
$r_i$.

(3) Suppose $G$ boundedly presented by $S_1$. Since $S_1\subset
S_1^m\subset S_1^m$, by (2) $G$ is boundedly presented by $S_1^m$.
Since $S_1^m\subset S_2^n\subset S_1^{mm'}$, by (2) again, $G$ is
boundedly presented by $S_2^n$. By (1), $G$ is boundedly presented
by $S_2$. Note that the hypothesis are symmetric in $S_1$ and
$S_2$, so that we do not need to prove the converse.\end{proof}

\begin{rem}
It is not true that if $G$ is boundedly presented by $S_1$ and
$S_2\supset S_1$, then $G$ is necessarily boundedly presented by
$S_2$. Indeed, let $G$ be equal to the infinite cyclic group $\Z$,
set $u_n=n!$, and let $S=\{u_n:\;n\ge 1\}$. Then $S$ generates
$G$, which is finitely presented. However, $G$ is not boundedly
presented by $S$.

Suppose the contrary, i.e. that there is a family of relations
between $u_k$'s of length $\le n-1$, with $n\ge 5$. We can add to
these relations the relations $[u_k,u_\ell]$, and write all other
ones as words of minimal length in the free \textit{abelian} group
generated by all $u_k$'s. Now consider a relation $r$ involving
$u_n$, which is not a commutator. Permuting the letters in $r$ if
necessary, we can write $r=r_1r_2r_3$, where $r_1$ involves
$u_k$'s for $k<n$, $r_2$ involves $u_n$'s, and $r_3$ involves
$u_k$'s for $k>n$. Then, in $\Z$, $|r_1|<n!$ while $r_2r_3$ is
divisible by $n!$. It follows that $r_1=r_2r_3$ in $\Z$. Now
$r_2=r_3^{-1}$ in $\Z$, but $|r_2|<(n+1)!$ and $r_3$ is divisible
by $(n+1)!$ in $\Z$. This implies $r_2=1$. Since $r$ has minimal
length, this implies that $u_n$ does not appear in $r$, a
contradiction.\label{rem:Z_not_bp_by_big_subset}\end{rem}

\begin{lem}
Let $G$ be boundedly presented by a symmetric generating subset
$S$. Let $N$ be a normal subgroup. Suppose that $N$ is generated,
as a normal subgroup, by $N\cap S^n$ for some $n$. Then $G/N$ is
boundedly presented by the image of $S$.\label{lem:bp_to_quot}
\end{lem}

\begin{lem}
Let $G$ be a group, and $N$ a normal subgroup; denote $p:G\to G/N$
the natural projection. Suppose that $N$ is boundedly presented by
a subset $T$. Let $T'\subset G$ be such that $p(T')=T$. Then $G$
is boundedly presented by $S=N\cup T'$.\label{lem:bp_from_quot}
\end{lem}
\begin{proof} Let $(r_i)$ be defining relations of bounded length for
$(G/N,T)$. 
 Let $\tilde{G}$ be defined by $T\cup S$ as set of
generators, and subject to the following relations:
\begin{itemize}
    \item[(i)] all possible relations obtained by lifting relations
        $r_i(t_1,\dots,t_n)=1$, giving relations of the form
        $r_i(\tilde{t}_1,\dots,\tilde{t}_n)=u$ (with $u\in N$),
    \item[(ii)] relations of the form $tut^{-1}=v$ for $t\in T'\cup
    T'^{-1}$ and $u,v\in N$,
    \item[(iii)] relations of bounded length defining $(N,T)$.
\end{itemize}
Let us show that $p:\tilde{G}\to G$ is bijective. Note that it
follows from the relations that $N$ can considered as a normal
subgroup of $\tilde{G}$. Let $u_1\dots u_n$ belong to $\tilde{G}$,
such that $u_1\dots u_n=1$ in $G$. Then $u_1\dots u_n=1$ in $G/N$,
so that we can write $u_1\dots u_n=\prod(g_km_kg_k^{-1})$, where
$m_k$'s are elements among $r_i$'s. Write $m_k=\rho_kv_k$, where
$\rho_k$ is a relation of $G$ from (i) and $v_k\in N$. Then
$\prod(g_km_kg_k^{-1})=\prod(g_k\rho_kg_k^{-1})v$ for some $v\in
N$ (this holds in the group only subject to the relations (ii)).
Thus in $\tilde{G}$, $u_1\dots u_n=v$. Hence $v$ is in the kernel
of $p$, but $p$ is injective in restriction to $N$. Accordingly
$v=1$ and $u_1\dots u_n=1$ in $\tilde{G}$.\end{proof}

\begin{lem}
Consider an extension $1\to N\to G\to Q\to 1$. Let $T$ resp. $N$
boundedly present $W$ resp. $Q$. Let $W'$ be a subset of $G$ whose
projection into $Q$ contains $W$. Suppose in addition that $G$ has
a Hausdorff topological group structure such that both $T$ and
$W'$ are compact and $T$ has non-empty interior in $N$. Then $G$
is boundedly presented by $S=T\cup
W'$.\label{lem:ext_cpt_pres_top}
\end{lem}
\begin{proof} Let $(r_i)$ be defining relations of bounded length for
$(Q,W)$. Let $\tilde{G}$ be defined by $T\cup S$ as set of
generators, and subject to the following relations:
\begin{itemize}
    \item[(i)] relations obtained by lifting relations
        $r_i(t_1,\dots,t_n)=1$, giving relations of the form
        $r_i(\tilde{t}_1,\dots,\tilde{t}_n)=u$ (with $u\in N$ expressed
        as a word in letters in $T\cup T^{-1}$),
    \item[(ii)] relations of the form $wuw^{-1}=v$ for $w\in W'\cup
    W'^{-1}$, $u\in T$, $v\in N$.
\end{itemize}
Observe that the relations in (i) are of bounded length: if $m$ is
a bound to the length of $r_i$'s, then the elements $u$'s
appearing in (i) belong to the intersection of $N$ with the
$m$-ball of $G$ with respect to $W'$, which is compact, hence is
contained in the $m'$-ball of $N$ with respect to $T$ for some
$m'$. Similarly, the $v$ appearing in (ii) has bounded length, so
that relators in (ii) also have bounded length.

Let us show that $p:\tilde{G}\to G$ is bijective. As in the
preceding proof, it follows from the relations that the subgroup
$\tilde{N}$ of $\tilde{G}$ generated by $T$ is normal in
$\tilde{G}$. Let $u_1\dots u_n$ belong to $\tilde{G}$, such that
$u_1\dots u_n=1$ in $G$. As in the preceding proof, there exists
$v\in\tilde{N}$ such that $u_1\dots u_n=v$ in $\tilde{G}$. Now it
follows from the relations (iii) that $p$ is injective in
restriction to $\tilde{N}$. Accordingly $v=1$ and $u_1\dots u_n=1$
in $\tilde{G}$.\end{proof}

\section{Compactly presented groups}\label{cpg}

\begin{defn}
Let $G$ be a topological group. We say that $G$ is compactly
presented if $G$ is Hausdorff and there is a compact subset
$S\subset G$ such that the (abstract) group $G$ is boundedly
presented by $S$.
\end{defn}

The following is a standard application of compactness and of the
Baire Theorem.
\begin{lem}
Let $G$ be a locally compact group, and $S$ a compact generating
subset. Let $K\subset G$ be compact. Then $K\subset S^n$ for
large~$n$.\qed
\end{lem}

This has the following immediate consequence.

\begin{lem}If a locally compact group $G$ is compactly presented,
then it is boundedly presented by all its compact generating
subsets.\qed
\end{lem}

\begin{rem}
This is not true without the assumption that $G$ is locally
compact. Indeed, endow $\Z$ with the $p$-adic topology
$\mathcal{T}_p$. Set $u_n=n!$ and $u_\infty=0$. Then
$\{u_n\,:\;n\le\infty\}$ is a compact subset of
$(\Z,\mathcal{T}_p)$. But we have shown in Remark
\ref{rem:Z_not_bp_by_big_subset} that $\Z$ is boundedly presented
by $\{u_n\,:\;n\le\infty\}$, while it is boundedly presented by
$\{1\}$.
\end{rem}

\begin{prop}
Let $G$ be a locally compact group which is an extension of
compactly presented groups. Then $G$ is compactly presented.\label{ext}
\end{prop}
\begin{proof} This follows from Lemma \ref{lem:ext_cpt_pres_top} using the
fact that, if $G$ is a locally compact group and $N$ a closed
normal subgroup, and if $p:G\to G/N$ denotes the projection, then
for every compact subset $W$ of $G/N$ there exists a compact
subset $W'$ of $G$ such that $p(W')=W$ (this is not true for
general topological groups).\end{proof}


\begin{lem}
Let $G$ be a topological group. Let $U$ be an open, symmetric
generating subset containing~$1$. Consider the group $\tilde{G}$
presented with $\tilde{U}=U$ as set of generators, and with
relations $gh=k$ for $g,h,k\in U$ whenever this relation holds in
$G$. Then $\tilde{G}$ has a unique topology such that the natural
projection $p:\tilde{G}\to G$ is continuous and $\tilde{U}$ is a
neighbourhood of $1$ in $\tilde{G}$ such that $p$ induces a
homeomorphism $\tilde{U}\to U$. Moreover, $p$ is open with
discrete kernel, and $\tilde{G}$ is Hausdorff whenever $G$
is.\label{lem:covering}
\end{lem}
\begin{proof} Recall that a group topology is characterized by nets
converging to 1. It is easily seen that for a topology on
$\tilde{G}$ satisfying the conditions of the lemma, a net $(x_i)$
converges to 1 if and only if eventually $x_i\in\tilde{U}$ and
$p(x_i)\to 1$. This proves the uniqueness.


Now let us construct this topology. We define
$\Omega\subset\tilde{G}$ to be open if, for every $x\in\Omega$,
the set $p(x^{-1}\Omega\cap\tilde{U})$ contains a neighbourhood of
$1$ in~$G$. This defines a topology: trivially, $\emptyset$ and
$\tilde{G}$ are open and the class of open sets is stable under
unions. Let $\Omega_1$ and $\Omega_2$ be open. Then, for every
$x\in\Omega_1\cap\Omega_2$,
$p(x^{-1}\Omega_1\cap\Omega_1\cap\tilde{U})=p(x^{-1}\Omega_1\cap\tilde{U})\cap
p(x^{-1}\Omega_2\cap\tilde{U})$ (due to the injectivity of $p$ on
$\tilde{U}$), so that $\Omega_1\cap\Omega_2$ is open.

Observe that, for such a topology, $\tilde{U}$ is open, and
$x_i\to x$ if and only if $x^{-1}x_i\to 1$, while $x_i\to 1$ if
and only if it satisfied the condition introduced at the beginning
of the proof. In particular, $x_i\to 1$ if and only if
$x_i^{-1}\to 1$, and, if  $x_i\to 1$ and $y_i\to 1$, then
$x_iy_i\to 1$. It remains to check that the topology is a group
topology.

We first claim that, for every net $(x_i)$ in $\tilde{G}$, if
$x_i\to 1$ and $g\in\tilde{G}$, then $y_i=gx_ig^{-1}\to 1$. This
is proved by induction on the length of $g$ as a product of
elements of $\tilde{U}$. This immediately reduces to the case when
$g\in\tilde{U}$. In this case, for large $i$, $p(gx_i)$ and
$p(gx_ig^{-1})=p(y_i)$ belong to $U$. Moreover, $gx_i\to g$, hence
eventually belongs to $\tilde{U}$. Since the relation
$(gx_i)g^{-1}=\tilde{y}_i$ (where $\tilde{y}_i$ denotes the
preimage of $p(y_i)$ in $\tilde{U}$) holds in $G$ and
$gx_i\in\tilde{U}$ for large $i$, by the definition of
$\tilde{G}$, we obtain that $(gx_i)g^{-1}=\tilde{y}_i$ in
$\tilde{G}$, so that $y_i=\tilde{y}_i\in\tilde{U}$ for large~$i$.
Now $p(y_i)\to 1$, and thus $y_i\to 1$.

Now let us prove that if $x_i\to x$, then $x_i^{-1}\to x^{-1}$.
Indeed, $x^{-1}x_i\to 1$, and by conjugating by $x$, we obtain
$x_ix^{-1}\to 1$. As observed above, we can take the inverse, so
that $xx_i^{-1}\to 1$, i.e. $x_i^{-1}\to x^{-1}$.

Now let us prove that if $x_i\to x$ and $y_i\to y$, then
$x_iy_i\to xy$. Indeed,
$(xy)^{-1}x_iy_i=y^{-1}x^{-1}x_iy_i=y^{-1}(x^{-1}x_iy_iy^{-1})y$.
By the combining the observations above, $y_iy^{-1}\to 1$, hence
$x^{-1}x_iy_iy^{-1}\to 1$, hence $y^{-1}(x^{-1}x_iy_iy^{-1})y\to
1$. So this is a group topology.

It is immediate that $p$ is continuous. Since
$\tilde{U}\cap\text{Ker}(p)=\{1\}$, the kernel is discrete.
Moreover, it is immediate from the definition that $p$ induces a
homeomorphism of $\tilde{U}$ onto $U$. Since these are
neighbourhoods of 1 in $\tilde{G}$ and $G$, this implies that $p$
is open. If $G$ is Hausdorff, then it is immediate from the
definition that $\tilde{G}-\{1\}$ is open, so that $\tilde{G}$ is
Hausdorff.\end{proof}

As a particular case, we get

\begin{prop}
Let $G$ be a compactly generated Hausdorff group. Then there
exists a compactly presented group $\tilde{G}$ with a normal
discrete subgroup $N$ such that $G\simeq\tilde{G}/N$. Moreover, if
$G$ is locally compact, then so is $\tilde{G}$.\label{cov}
\end{prop}

\begin{prop}
Let $G$ be a Hausdorff topological group, and $K$ a compact normal
subgroup. Then $G$ is compactly presented if and only if $G/K$ is.
\end{prop}
\begin{proof} This immediately follows from Lemmas \ref{lem:bp_from_quot}
and \ref{lem:bp_to_quot}.\end{proof}

The following lemma is due to Macbeath and Swierczkowski \cite{MaS}.

\begin{lem}
Let $G$ be a topological Hausdorff group, and let $H$ be a closed
subgroup. Suppose that $H$ is cocompact in $G$, i.e. there exists a compact subset $K\subset G$
such that $G=HK=\{hk:\;h\in H,k\in K\}$. Then $G$ is compactly
generated if and only if $H$ is.\label{lem:cpt_eng_cocpt}
\end{lem}
\begin{proof} Clearly, if $H$ is generated by a compact subset $S$, then
$G$ is generated by $S\cup K$. Conversely, suppose $G$ generated
by a compact symmetric subset $S$; suppose also that $1\in K$. Set
$T=KSK^{-1}\cap H$ where $KSK^{-1}=\{k_1sk_2^{-1}:\;k_1,k_2\in
K,\,s\in S\}$. Then $T$ is compact, we claim that it generates
$H$.

Let $g$ belong to $G$, and write $g=s_1\dots s_n$ with all $s_i\in
S$. Set $k_0=1$, and define by induction $k_i\in K$, for $1\le
i\le n-1$ by: $k_{i}$ is chosen in $K$ so that
$k_{i-1}s_ik_i^{-1}\in H$ (it exists because $HK=G$). Define
$k_n=1$. Then $g=\prod_{i=1}^nk_{i-1}g_ik_i^{-1}$. Since
$g=\prod_{i=1}^{n-1}k_{i-1}g_ik_i^{-1}\in H$,
$k_{n-1}g_nk_n^{-1}\in H$ if (and only if) $g\in H$. We thus
obtain that $H$ is generated by $T$.\end{proof}

\begin{rem}
The proof also implies that, for $g\in H$, the length of $g$ with
respect to $T$ is bounded by the length of $g$ with respect to
$S$. Hence, in $H$, the word length with respect to $T$ and the length induced by the
word length of $G$ are equivalent, i.e. the embedding of $H$ into $G$ (endowed with these compact generating subsets) is a quasi-isometry.
\end{rem}

\begin{prop}
Let $G$ be a locally compact group, and let $H$ be a closed cocompact
subgroup. Suppose that there exists a compact subset $K\subset G$
such that $G=HK=\{hk:\;h\in H,k\in K\}$. Then $G$ is compactly
presented if and only if $H$ is.
\end{prop}
The reader can try to prove this directly, but the proof is tedious and technical, especially the direction $\Rightarrow$. This will follow from Lemma \ref{lem:cpt_eng_cocpt}, the subsequent remark, as well of results from the next section (Corollary \ref{qinv}).

\section{Simple connectedness}\label{rips}

\subsection{The Rips complex}

We denote by $X_{\textnormal{top}}$ the topological realization of a simplicial complex $X$.

\begin{lem}
Let $X$ be a simplicial complex, with a vertex as base-point
$x_0$. Let $\gamma$ be a loop in $X_{\textnormal{top}}$
based on $x_0$. Then $\gamma$ is homotopic to a combinatorial
loop, i.e. a loop which is a sequence of consecutive edges
travelled with constant speed.\label{lem:path_X1}
\end{lem}
\begin{proof} The reader can check it as an exercise, or refer to \cite[Chap.3, Sec.6,
Lemma~12]{Spa}.\end{proof}

\medskip

Let $\Gamma$ be a graph, and $x_0$ a base-point. Consider a
combinatorial loop based on $x_0$, represented as a sequence
$(x_0,x_1,\dots,x_n)$ where $x_n=x_0$ and $x_{i-1}x_i$ is an edge
for $i=1,\dots,n$. On the set of such loops, consider the
equivalence relation generated by
$$(x_0,x_1,\dots,x_n)\sim(x_0,\dots,x_i,u,x_i,x_{i+1},\dots,x_n)$$ whenever $\{u,x_i\}$ is an edge,
which we call ``graph-homotopic". Besides, given two paths
$y=(y_1,\dots,y_n)$ and $z=(z_1\dots,z_m)$, the composition $yz$
is defined if $y_n=z_1$ and denotes the path
$(y_1,\dots,y_n,z_2\dots,z_m)$; similarly $y^{-1}$ denotes
$(y_n,\dots,y_1)$.

Let $X$ be a simplicial complex. Consider paths joining given
points $x,x'$. On this set, consider the equivalence relation
generated by graph homotopies and $$(x=x_0,\dots,x_n=x')\sim
(x=x_0,\dots,x_i,u,x_{i+1},\dots,x_n=x')$$ whenever
$\{x_i,u,x_{i+1}\}$ is a 2-simplex; two paths equivalent for this
equivalence relation are called combinatorially homotopic.

\begin{lem}
Let $X$ be a simplicial complex. Let two combinatorial paths be
homotopic (in the topological sense). Then they are
combinatorially homotopic.\label{lem:homot_comb_homot}
\end{lem}
\begin{proof} This is \cite[Chap.3, Sec.6, Theorem~16]{Spa}.\end{proof}

\begin{lem}
Let $X$ be a simplicial complex, $X^1$ its 1-skeleton, and $x_0$ a
base-point which is a vertex. Let $x=(x_0,x_1,\dots,x_n=x_0)$ be a
loop. Suppose that $x$ is, as a based loop in
$X_{\textnormal{top}}$, homotopically trivial. Then $x$ is
graph-homotopic to a loop of the form $\prod_{i=1}^m
y_ir_iy_i^{-1}$, where $y_i$ is a path from $x_0$ to a point
$z_i$, and $r_i$ is a loop based on $z_i$, such that all vertices
of $r_i$ belong to a common simplex in $X$, and such that the
length of $r_i$ is 3, i.e. $r_i$ is of the form
$(z_i,z'i,z''_i,z_i)$.\label{lem:graphomotopic_simplices}
\end{lem}
\begin{proof} This is a reformulation of Lemma \ref{lem:homot_comb_homot}.
Indeed, suppose that $\{x_{i-1},x_i,x_{i+1}\}$ is a 2-simplex.
Then $(x_0,\dots,x_n=x_0)$ is graph-homotopic to
$$u(x_{i-1},x_i,x_{i+1})u^{-1}(x_0,\dots,x_{i-1},x_{i+1},\dots,x_n=x_0),$$
where $u=(x_0,\dots,x_{i-1},x_{i+1})$. It then suffices to iterate
this process.\end{proof}

\medskip

Let $X$ be a metric space. In all what follows, we consider $X$
with the discrete topology (we study metric at large scale).
Define the Rips complex of $X$ as follows: for $t\in\R_+$,
$R_t(X)$ is the simplicial complex with $X$ as set of vertices,
and $(x_1,\dots,x_d)$ is a simplex if $d(x_i,x_j)\le t$ for all
$i=1,\dots,d$. It was used in particular by Gromov's monograph on hyperbolic groups \cite[Section~1.7]{G}.

Let $G$ be a group and $S$ a symmetric generating subset. Viewing
$G$ as a metric space (for the word length), we can define its
Rips complex.

\begin{lem}
The embedding of $G$ into $R_t(G)$ is a quasi-isometry.\label{ripsqi}
\end{lem}

\begin{prop}
Let $G$ be a group endowed with a generating subset $S$. The following conditions are equivalent.\begin{itemize}\item The
group $G$ is boundedly presented by $S$.\item
$R_t(G)_{\textnormal{top}}$ is simply connected for sufficiently
large $t$. \item $R_t(G)_{\textnormal{top}}$ is simply connected
for some $t$.\end{itemize}\label{bpresrips}
\end{prop}
\begin{proof} Suppose that $R_m(G)_{\textnormal{top}}$ is simply connected.
Since the distance is integer-valued, we can suppose that $m$ is
integer (and $m\ge 1$: if $m=0$ the only possibility is
$G=\{1\}$). We endow $G$ with the generating subset $S^m$, so that
its Cayley graph coincides with the 1-skeleton of
$R_m(G)_{\textnormal{top}}$. Replacing $S$ by $S^m$ if necessary,
we assume that $m=1$.

Let $s_1,\dots,s_m$ belong to $S$, such that $s_1\dots s_n=1$.
Consider the combinatorial path $(1,s_1,s_1s_2,\dots,s_1\dots
s_{n-1},1)$: the corresponding path in $R_1(G)_{\textnormal{top}}$
is homotopically trivial. It follows from Lemma
\ref{lem:graphomotopic_simplices} that, in the free group
generated by $S$ (only subject to the relations of the form
$ss^{-1}=1$), $s_1\dots s_n=\prod_{i=1}^kg_ir_ig_{i}^{-1}$ for
some elements $g_i$, and some $r_i$ such that, if we write
$r_i=m_{i1}m_{i2}m_{i3}$ with $m_{ij}\in S$, then $r_i=1$ in $G$.
Therefore, we obtain that $S$ is a defining subset of $G$.

\medskip

Conversely, suppose that $G$ is boundedly presented by $S$, by
relations of length $\le m$. Let us show that $R_{\lfloor
m/2\rfloor}(G)$ is simply connected. Consider a path $\gamma$ in
$R_{\lfloor m/2\rfloor}(G)_{\textnormal{top}}$, based at 1, and
let us show that it is homotopically trivial. By Lemma
\ref{lem:path_X1}, we can suppose that $\gamma$ is a combinatorial
path $$(1,s_1,s_1s_2,\dots,s_1s_2\dots s_{n-1},s_1s_2\dots
s_{n-1}s_n=1).$$ By the presentation of $G$, we can write, in the
free group generated by $S$ (only subject to the relations of the
form $ss^{-1}=1$), $s_1\dots s_n=\prod_i g_ir_ig_i^{-1}$, where
$r_i$ has length $\le m$. To each element $u=u_1\dots u_k$ of the
semigroup freely generated by $S$, corresponds a combinatorial
path $[u]=(1,u_1,u_1u_2,\dots u_1\dots u_k)$. Then $[s_1\dots
s_n]$ and $[\prod_i g_ir_ig_i^{-1}]$ are graph-homotopic: this is
a trivial consequence of the solution of the word problem in a
free group. Now observe that, clearly, if two combinatorial based
loops are graph-homotopic, then they are homotopic. Thus we are
reduced to prove that $[\prod_i g_ir_ig_i^{-1}]$ is homotopically
trivial. Clearly, this reduced to showing that each
$[g_ir_ig_i^{-1}]$ is homotopically trivial. But this is freely
homotopic to $[r_i]$, and this combinatorial loop, being of length
$\le m$, has all its vertices at distance $\le\lfloor m/2\rfloor$.
Hence it lies in the boundary of a simplex of $R_{\lfloor
m/2\rfloor}(G)$, and thus is homotopically trivial.\end{proof}

\medskip

Let $X$ be a metric space. A $r$-path in $X$ is a sequence
$(x_0,x_1,\dots,x_n)$ such that $d(x_i,x_{i+1})\le r$ for all $i$;
it is said to join $x_0$ and $x_n$. If $x_0=x_n$, it is called a
$r$-loop based on $x_0$.

We introduce the $r$-equivalence between $r$-paths: this is the
equivalence relation generated by
$(x_0,\dots,x_n)\sim(x_0,\dots,x_i,y,x_{i+1},x_n)$.

\begin{defn}
The metric space $X$ is called coarsely connected if, for some
$r$, every two points in $X$ are joined by a $r$-path.

The metric space $X$ is called coarsely simply connected if it is
coarsely connected, and, for every $r$, there exists $r'\ge r$
such that every $r$-loop is $r'$-homotopic to the trivial loop.
\end{defn}

\begin{rem}
Consider a geodesic circle $C_R$ of length $R$. Let a $r$-path
$\gamma$ go round $C_R$. Then the reader can check that $\gamma$
is $r$-homotopic to the trivial path if and only if $R\le 3r$. In
particular, a bunch of circles of radius tending to $\infty$ if
not coarsely simply connected. Removing one point (other than the
base-point) in each of these circles, we obtain a simply connected
(and even contractible) metric space which is not coarsely simply
connected.
\end{rem}

\begin{prop}
Let $X$ be a metric space. The following conditions are equivalent:

(1) $X$ is coarsely simply connected,

(2)  For some $t_0$, $R_{t_0}(X)$ is connected, and for every $t$,
there exists $t'\ge t$ such that every loop in $R_t(X)$ is
homotopically trivial in $R_{t'}(X)$.

(3) For some $r$, $R_r(X)$ is simply connected.\label{coarips}
\end{prop}
\begin{proof} (1)$\Rightarrow$(2) Suppose that $X$ is coarsely simply
connected. Clearly, $R_{t_0}(X)$ is connected for some $t_0$.
Consider a loop in $R_r(X)$. By Lemma \ref{lem:path_X1}, it is
homotopic to a combinatorial loop, defining a $r$-path on $X$. For
some $r'\ge r$, this loop is $r'$-equivalent to the trivial loop.
Thus the loop is homotopically trivial in $R_{r'}(X)$.

(2)$\Rightarrow$(1) Suppose that (2) is satisfied. Then $X$ is
$t_0$-connected. Fix $t\ge t_0$, and take $t'$ as in (2). Consider
a $t$-loop in $X$. Then the corresponding loop in $R_{t'}(X)$ is
homotopically trivial. By Lemma \ref{lem:homot_comb_homot}, we
obtain that it is $t'$-equivalent to the trivial loop.

(3)$\Rightarrow$(2). Consider a $t$-loop in $X$. Set
$r'=\max(r,t)$. Since $X$ is $r$-connected, this loop is
$r'$-equivalent to a $r$-loop, defining a combinatorial loop on
$R_r(X)$. Since $R_r(X)$ is simply connected and using Lemma
\ref{lem:homot_comb_homot}, this $r$-loop is $r$-equivalent to the
trivial loop. Thus our original loop is $r'$-equivalent to the
trivial loop.

(2)$\Rightarrow$(3). Fix $t\ge t_0$, and $t'$ as in (2). First
note that $R_{t'}(X)$ is pathwise connected. Consider a loop in
$R_{t'}(X)$. It is homotopic to a combinatorial loop. Since $X$ is
$t$-connected, it is $t'$-equivalent to a $t$-loop. This one is,
by assumption, homotopically trivial in $R_{t'}(X)$. Thus
$R_{t'}(X)$ is simply connected.\end{proof}

From Lemma \ref{ripsqi}, Proposition \ref{bpresrips} and Proposition \ref{coarips}, we get

\begin{cor}
Let $G$ be a group generated by a subset $S$, endowed with the word metric. Then $G$ is boundedly presented by $S$ if and only $G$ is coarsely simply connected. In particular, a locally compact compactly generated group is compactly presented if and only if it is coarsely simply connected.
\end{cor}

\medskip


\medskip

\begin{prop}Being coarsely simply connected is a quasi-isometry
invariant.
\end{prop}
\begin{proof} Consider metric spaces $X,Y$, and functions $f:X\to Y$,
$g:Y\to X$ such that
$$\forall x,x'\in X,\;d(f(x),f(x'))\le Ad(x,x')+B\,;$$
$$\forall y,y'\in X,\;d(f(y),f(y'))\le \alpha d(y,y')+\beta\,;$$
$$\forall x\in X,\; d(g\circ f(x),x)\le C\,;\quad \forall y\in X,\; d(f\circ g(y),y)\le
\gamma.$$ Suppose that $Y$ is coarsely simply connected.

Then $Y$ is $r$-connected for some $r$. If $x,x'\in X$, then there
exists a $r$-path $(f(x)=y_0,\dots,y_n=f(x'))$. Set
$r'=\max(C,\alpha r+\beta)$. Then $(x,g\circ
f(x)=g(y_0),\dots,g(y_n)=g\circ f(x'),x')$ is a $r'$-path joining
$x$ and $x'$, so that $X$ is $r'$-connected.

Consider now a $\rho$-loop $x_0,\dots,x_n=x_0$ in $X$. Then
$f(x_0),\dots,f(x_n)$ is a $(A\rho+B)$-loop in $Y$. Since $Y$ is
coarsely simply connected, there exists $R\ge A\rho+B$ (depending
only on $A\rho+B$ and not on the loop) such that
$f(x_0),\dots,f(x_n)$ is $R$-equivalent to the trivial loop. Thus
$g\circ f(x_0)\dots g\circ f(x_n)$ is $(\alpha
R+\beta)$-equivalent to the trivial loop. Setting
$\rho'=C+\max(\alpha R+\beta,\rho)$, we obtain that
$(x_0,\dots,x_n)$ is $\rho'$-equivalent to $g\circ f(x_0)\dots
g\circ f(x_n)$,
hence is $\rho'$-equivalent to the trivial loop.\end{proof}





\begin{cor}
Among locally compact compactly generated groups, being compactly presented is a quasi-isometry invariant.\label{qinv}
\end{cor}

Given a metric space $X$ and a path $\gamma:[0,1]\to X$, its
diameter is defined as $\sup_{u,v\in [0,1]}d(\gamma(u)-\gamma(v)$.
We call a metric space weakly geodesic if there exists a function
$w:\R_+\to\R_+$ such that, whenever $x,y\in X$, they are joined by
a path of diameter $\le w(d(x,y))$.

\begin{prop}
Let $X$ be a weakly geodesic metric space. If $X$ is simply
connected, then it is coarsely simply
connected.\label{prop:simplyC_coarse_simplyC}
\end{prop}
\begin{proof} We can suppose that $w(r)\ge r$ for all $r$.

Clearly, $X$ is path-connected. Fix $r$, and consider a $r$-loop
$(x_0,x_1,\dots,x_n=x_0)$ in $X$. Interpolate it to obtain a path
$\gamma:[0,n]\to X$ such that $\gamma(i)=x_i$ for all
$i=0,\dots,n$, and such that, for all $i\le t,t'\le i+1$, we have
$d(\gamma(t),\gamma(t'))\le R=w(r)$.


The path $\gamma$ is homotopically trivial. Thus there is a
continuous function $h:[0,1]^2\to X$ such that, for all $t$,
$h(t,0)=\gamma(t)$ and $h(0,t)=h(1,t)=h(t,1)=x_0=x_n$ for all $t$.

Since $h$ is uniformly continuous, there exists $\eta$ such that,
for all $x,y,x',y'$, if $\max(|x-x'|,|y-y'|)\le\eta$, then
$d(h(x,y),h(x',y'))\le R$.

Using the assumption on the diameter, the $R$-path
$(x_0,x_1,\dots,x_n)$ is $R$-homotopic to a path of the form
$(\gamma(t_0),\dots,\gamma(t_m)$, with $0=t_0\le t_1\le\dots \le
t_n=1$, and $t_{j+1}-t_j\le\eta$ for all $j$.

Now
\begin{align*}
    (\gamma(t_0),\dots,\gamma(t_m)=(h(t_0,t_0),h(t_1,t_0),\dots,h(t_n,t_0))\\
    =(h(t_0,t_1),h(t_1,t_0),\dots,h(t_n,t_0))
\end{align*}
is $R$-equivalent to
$(h(t_0,t_1),h(t_1,t_1),h(t_2,t_0)\dots,h(t_n,t_0))$ (passing
through
$$(h(t_0,t_1),h(t_1,t_1),h(t_1,t_0),h(t_2,t_0)\dots,h(t_n,t_0))\,).$$
Iterating in a similar way, it is $R$-equivalent to
$(h(t_0,t_1),h(t_1,t_1),\dots,h(t_n,t_1))$. Iterating all the
process, we obtain that it is $R$-equivalent to
$$(h(t_0,t_1),h(t_1,t_1),\dots,h(t_n,t_1))=(x_0,x_0,\dots,x_0),$$
which is $R$-equivalent to the trivial loop.\end{proof}

\subsection{Coverings}

\begin{defn}
Let $X\to Y$ be a map between metric spaces. We call it a
$r$-metric covering if, for every $y\in Y$, such that the closed
$r$-balls $B'(x,r)$, $x\in f^{-1}(y)$ are pairwise disjoint, and
that, for every $x,x'\in X$ such that $d(x,x')\le n$, we have
$d(x,x')=d(f(x),f(x'))$ (in particular, $f$ is an isometry in
restriction to $(n/2)$-balls).
\end{defn}


\begin{prop}
Let $Y$ be $(r,R)$-simply connected. Let $X$ be another metric
space and $f:X\to Y$ be a $R$-covering. Then $f$ is
injective.\label{prop:covering_injective}
\end{prop}
\begin{proof} Suppose that $f(x)=f(x')=y_0$ for some $x_0,x_1\in X$.
Consider a $r$-path $(x=x_0,x_1,\dots,x_n=x')$ between $x$ and
$x'$. Write $y_i=f(x_i)$. Consider a $R$-combinatorial homotopy
between $(y_0,\dots,y_n=y_0)$ and the trivial loop $(y_0)$. Since
$f$ is a $R$-covering, the homotopy lifts in a unique way to $X$.
Thus $x=x'$.\end{proof}

\begin{prop}
Let $1\to N\to G\to Q\to 1$ be an exact sequence of locally
compact groups. Suppose that $G$ is compactly generated and that
$Q$ is compactly presented. Then $N$ is compactly generated as a
normal subgroup of $G$.\label{prop:ker_comp_pres_quot}
\end{prop}
\begin{proof} Let $S$ be a generating subset of $G$, whose interior
contains 1, and let $B_n$ denote the $n$-ball of $G$ relative to
$S$. Let $N_n$ be the normal subgroup of $G$ generated by $G\cap
B_n$. We must show that eventually $N=N_n$.

Endow $G/N_n$ with the generating subset $S_n$, defined as the
image of $S$ in $G/N_n$. Then the projection of $G/N_{2n+1}$ onto
$G/N$ is a $n$-covering. Since $G/N$ is coarsely simply connected,
this implies, by Proposition \ref{prop:covering_injective}, that
the projection $G/N_n\to G/N$ is injective for sufficiently large
$n$. Thus eventually $N_n=N$.\end{proof}

\section{Applications and examples}\label{appex}

\begin{prop}
Let $G$ be a locally compact group with finitely many connected
components. Then $G$ is compactly presented.
\end{prop}
\begin{proof} Clearly, $G$ is compactly generated. It is known (see \cite{MoZi}) that
$G$ has a compact normal subgroup $W$ such that $G/K$ is a Lie
group. By a result of Mostow \cite{Most}, it follows that $G$ has
a maximal compact subgroup $K$ such that $G/K$ is diffeomorphic to
a Euclidean space. Taking a $G$-invariant Riemannian metric on
$G/K$, we obtain that $G$ is quasi-isometric to $G/K$, which is
simply connected and geodesic. By Proposition
\ref{prop:simplyC_coarse_simplyC}, $G$ is coarsely simply
connected, hence is compactly presented.\end{proof}

\begin{rem}
When $G$ is a simply connected Lie group, it also follows from
Lemma \ref{lem:covering} that $G$ is compactly presented. Thus,
the (true) fact that every connected Lie group is compactly presented is
equivalent to each of the following assertions: (i) Every
connected Lie group has a finitely generated fundamental group;
(ii) every discrete, central subgroup in a connected Lie group is
finitely generated. Note that (i) also follows directly from
\cite{Most}.
\end{rem}

\begin{prop}
Let $G$ be a locally compact, compactly generated group of
polynomial growth. Then it is compactly presented.
\end{prop}
\begin{proof} By a result of Losert \cite{Losert} (relying on a result of
Gromov \cite{Gromov81}), $G$ lies in an iterate extension
compact-by-(connected Lie group)-by-(discrete virtually
nilpotent). Thus $G$ is compactly presented.\end{proof}

\begin{rem}
It can be shown (see \cite{Bre}) that a locally compact, compactly generated group
of polynomial growth is quasi-isometric to a simply connected
nilpotent Lie group.
\end{rem}

\begin{exe}\label{slh}
Let $\K$ be a non-Archimedean local field (e.g. $\K=\Q_p$). Define
the Heisenberg group $H_3(\K)$ as follows: as a topological space,
$H_3(\K)=\K^3$, and its group law is given by
$(x_1,y_1,z_1)(x_2,y_2,z_2)=(x_1+x_2,y_1+y_2,z_1+z_2+x_1y_2-x_2y_1)$.
The group $\SL_2(\K)$ acts on $H_3(\K)$ by $\begin{pmatrix}
  a & b \\
  c & d \\
\end{pmatrix}\cdot(x,y,z)=(ax+by,cx+dy,z)$. It is easy to check
that the semidirect product $\SL_2(\K)\ltimes H_3(\K)$ is
compactly generated. Its centre coincides with the centre
$\{0\}\times\{0\}\times\K$ of $H_3(\K)$, hence is not compactly
generated as a normal subgroup of $\SL_2(\K)\ltimes H_3(\K)$.
Thus, by Proposition \ref{prop:ker_comp_pres_quot}, the quotient
by the centre, namely $\SL_2(\K)\ltimes\K^2$, is compactly generated but
not compactly presented.
\end{exe}

\begin{prop}
Let $G$ be a Gromov-hyperbolic, compactly generated locally compact group. Then $G$ is compactly presented.
\end{prop}
\begin{proof}
Gromov then shows \cite[1.7.A]{G} that the Rips complex $R_t(G)_{\textnormal{top}}$ is contractible for $t$ large enough. So Proposition \ref{coarips} applies.
\end{proof}

\appendix

\section{Some locally compact groups not quasi-isometric to homogeneous locally finite graphs}\label{secho}

We here justify the following well-known result

\begin{prop}
Let $G$ is a simply connected graded nilpotent Lie group with no lattice (i.e.\ the Lie algebra has no form over the rationals, see \cite[Theorem~2.12]{Rag}). Then 
$G$ is not quasi-isometric to a homogeneous graph.
\end{prop}

An example of a simply connected nilpotent Lie group with no lattice is given in \cite[Remark~2.14]{Rag}; the examples given there are nilpotent of class two and therefore are obviously graded.


\begin{proof}Trofimov \cite{Tro} proved, relying on Gromov's polynomial growth theorem \cite{Gromov81} that a homogeneous (i.e.~vertex-transitive) graph with polynomial growth is quasi-isometric to a finitely generated torsion-free nilpotent group $\Gamma$. So $G$ is quasi-isometric to the Malcev closure $N$ of $\Gamma$ (a simply connected nilpotent Lie group in which $\Gamma$ is a lattice). Then $M$ having a lattice, it has a form over the rationals, and therefore so does the associated graded nilpotent group $M_g$. Then by results of Pansu \cite{Pan}, $M_g$ is isomorphic to $G$. We thus get a contradiction.\end{proof}

Many other Lie groups without cocompact lattices are likely not to be quasi-isometric to any homogeneous graphs, including
\begin{itemize}
\item Semidirect products $G_\alpha=\R^2\rtimes\R$ with the action by diagonal matrices with coefficients $(e^t,e^{\alpha t}$, excluding the special cases $\alpha=-1,0,1$. (The case $\alpha>0$ should be easier to tackle as then $G$ is negatively curved.)
\item Semidirect products $\R^n\rtimes\textnormal{GL}(n,\R)$ or $\R^n\rtimes\SL(n,\R)$ for $n\ge 2$.
\end{itemize}

Note that in general, a space is quasi-isometric to a homogeneous graph of bounded degree if and only if it is quasi-isometric to a totally discontinuous compactly generated locally compact group: indeed, any such graph is quasi-isometric to its isometry group, and conversely for any totally discontinuous compactly generated locally compact group $G$, if $K$ is an open compact subgroup then $G/K$ carries a left-invariant bounded degree graph structure (this elementary construction has been known for decades but the older reference I could find is \cite[Chap.~11, p.~150]{Mon}).



\begin{thebibliography}{KM98b}


\bibitem[Ab]{Ab1} H. {\sc Abels}. \newblock
{\em Kompakt definierbare topologische Gruppen}. (German)                             
Math. Ann. 197 (1972), 221--233. 



\bibitem[Ab2]{Ab} H. {\sc Abels}.
{\em Finite presentability of $S$-arithmetic groups. Compact
presentability of solvable groups}.
 Lecture Notes in Math. {\bf 1261}, Springer, 1987.


\bibitem[Behr]{Behr} Helmut {\sc Behr}.
    \newblock {\em Über die endliche Definierbarkeit verallgemeinerter Einheitgruppen. II}.
    \newblock Inv. Math. {\bf 4} (1967), 265-274.

\bibitem[Bre]{Bre} E. {\sc Breuillard}. {\em Geometry of locally compact groups of polynomial growth and shape of large balls}. Preprint 2007, arXiv:0704.0095 v1.

\bibitem[Dy]{Dy} W. {\sc van Dyck}. {\em Gruppentheoretische Studien}.                 
Math. Ann. 20 (1882), 1--44.

\bibitem[Gr81]{Gromov81} Mikhael {\sc Gromov}.
    \newblock {\em Groups of polynomial growth and expanding maps}.
    \newblock Publ. Math. Inst. Hautes \'Etudes Sci. {\bf 53} (1981), 53-73.

\bibitem[Gr87]{G} M. {\sc Gromov}. ``Hyperbolic groups". In {\it Essays in group theory} 75--263 (Edt. S. Gersten) Math. Sci. Res. Inst. Publ., 8, Springer, New York, 1987.

\bibitem[Gr93]{Gromov93} Mikhael {\sc Gromov}.
    \newblock {\em Asymptotic invariants of infinite groups.}
    \newblock In: ``Geometric group theory. Volume 2".
    \newblock London Math. Soc. Lecture Notes 182, Cambridge Univ.
    Press, 1993.

\bibitem[Kn]{Kn} M.~{\sc Kneser}.
	\newblock {\em Erzeugende und Relationen verallgemeinerter Einheitengruppen}. 
	\newblock Crelle's Journal 214/15, 345--349 (1964). 

\bibitem[Lo]{Losert} Viktor {\sc Losert}.
    \newblock {\em On the structure of groups with polynomial growth.}
    \newblock Math. Z. {\bf 195} (1987) 109-117.

\bibitem[MaS]{MaS} A. M. {\sc Macbeath}, S. {\sc Swierczkowski}. {\em On the set of generators of a subgroup}. Indag. Math. 21 (1959), 280--281.

\bibitem[Mon]{Mon} N.~{\sc Monod}. ``Continuous bounded cohomology of locally compact groups", Lecture Notes in Math. 1758, Springer 2001.


\bibitem[MonZ]{MoZi} D. {\sc Montgomery}, L. {\sc Zippin}. {\em Topological transformation groups}. Interscience Publishers,  New York, 1956.

\bibitem[Most]{Most} G. D. {\sc Mostow}.
    \newblock {\em Self-adjoint groups}.
    \newblock Ann. of Math. {\bf 62} (1967), 44-55.

\bibitem[Pan]{Pan} P.~{\sc Pansu} {\em M\'etriques de Carnot-Carath\'eodory et quasiisom\'etries des espaces sym\'etriques de rang un}. (French) 
Ann. Math., II. Ser. {\bf 129}, No.1, 1--60 (1989).

\bibitem[Rag]{Rag} M. S. {\sc Raghunathan}. ``Discrete subgroups of Lie groups". Springer-Verlag, 1972. 

\bibitem[Spa]{Spa} Edwin H. {\sc Spanier}.
    \newblock ``Algebraic topology".
    \newblock McGraw Hill 1966.

\bibitem[Tro]{Tro} V.~{\sc Trofimov} {\em Graphs with polynomial growth}. Math. USSR Sb. 51 (1985) 405--417.


\end{thebibliography}
\end{document}